\documentclass[reqno, 12pt]{amsart}
\usepackage{amsmath,amsxtra,amssymb,latexsym, amscd,amsthm}
\usepackage[utf8]{inputenc}
\usepackage{enumerate}
\usepackage[mathscr]{euscript}
\usepackage{mathrsfs}
\usepackage[english]{babel}
\usepackage{color}

\newtheorem*{DLRTheorem}{Theorem}
\newtheorem {theorem}{Theorem}[section]
\newtheorem {proposition}[theorem]{Proposition}

\newtheorem {definition}[theorem]{Definition}

\def\ar{a\kern-.370em\raise.16ex\hbox{\char95\kern-0.53ex\char'47}\kern.05em}
\def\ees{{\accent"5E e}\kern-.385em\raise.2ex\hbox{\char'23}\kern-.08em}
\def\eex{{\accent"5E e}\kern-.470em\raise.3ex\hbox{\char'176}}
\def\AR{A\kern-.46em\raise.80ex\hbox{\char95\kern-0.53ex\char'47}\kern.13em}
\def\EES{{\accent"5E E}\kern-.5em\raise.8ex\hbox{\char'23 }}
\def\EEX{{\accent"5E E}\kern-.60em\raise.9ex\hbox{\char'176}\kern.1em}
\def\ow{o\kern-.42em\raise.82ex\hbox{
\vrule width .12em height .0ex depth .075ex \kern-0.16em \char'56}\kern-.07em}
\def\OW{O\kern-.460em\raise1.36ex\hbox{
\vrule width .13em height .0ex depth .075ex \kern-0.16em \char'56}\kern-.07em}
\def\UW{U\kern-.42em\raise1.36ex\hbox{
\vrule width .13em height .0ex depth .075ex \kern-0.16em \char'56}\kern-.07em}
\def\DD{D\kern-.7em\raise0.4ex\hbox{\char '55}\kern.33em}

\def\Limsup{\mathop{{\rm Lim}\,{\rm sup}}}

\title{The radius of metric regularity at infinity}

\author{MINH T\`UNG NGUY\EEX N} 
\address[Minh T\`ung Nguy\eex n]{Faculty of Data Science of Business, Ho Chi Minh University of Banking, Ho Chi Minh City,Vietnam}
\email{tungnm@hub.edu.vn}

\author{TI\EES N-S\OW N PH\d{A}M}
\address[Ti\ees n-S\ow n Ph\d{a}m]{Department of Mathematics, Dalat University, 1 Phu Dong Thien Vuong, Dalat, Vietnam}
\email{sonpt@dlu.edu.vn}

\date{ \today}

\date{\today}



\begin{document}

\begin{abstract} 
This paper, in the setting {\em at infinity}, presents some relationships between the modulus of metric regularity and 
the radius of (strong) metric regularity that gives a measure of the extent to which a set-valued mapping can be perturbed before (strong) metric regularity is lost. The results given here can be viewed as versions at infinity of \cite[Theorem~1.5]{Dontchev2003} and \cite[Theorem~4.6]{Dontchev2004}.
\end{abstract}

\maketitle

\section{Introduction}

Metric regularity is one of the central concepts of variational analysis and plays a crucial role in the study of some basic problems of analysis such as the existence and behavior of solutions of generalized equations. For more details, including historical remarks, we refer the reader to the comprehensive monographs \cite{Borwein2005, Dontchev2009, Ioffe2017, Mordukhovich2006, Mordukhovich2018, Penot2013, Rockafellar1998}.

This paper deals with metric regularity properties from a new perspective, which addresses their stability {\em at infinity.}
Attention will be paid to the radius of metric regularity theorem that was initiated by Dontchev, Lewis and Rockafellar in \cite{Dontchev2003}. 

Let $X$ and $Y$ be Banach spaces, $F \colon X \rightrightarrows Y$ be a set-valued mapping, and 
$(\bar{x}, \bar{y}) \in \mathrm{gph}F.$ Recall that the mapping $F$ is {\em metrically regular} at $(\bar{x}, \bar{y})$ if there is a constant $\kappa > 0$ together with neighborhoods $U$ of $\bar{x}$ and $V$ of $\bar{y}$ such that 
\begin{eqnarray*}
\mathrm{dist}\big(x, F^{-1}(y) \big) & \leqslant & \kappa\, \mathrm{dist}\big(y, F(x) \big) \quad \textrm{ for all } \quad (x, y) \in U \times V,
\end{eqnarray*} 
where $\mathrm{dist}(\cdot, \cdot)$ denotes the usual distance function and $F^{-1} \colon Y \rightrightarrows X$ stands for the inverse mapping of $F.$
The infimum of the set of values $\kappa$ for which this holds is called the {\em modulus of metric regularity} at $(\bar{x}, \bar{y})$ and denoted by $\mathrm{reg} F (\bar{x}, \bar{y}).$ The mapping $F$ is {\em strongly metrically regular} at $(\bar{x}, \bar{y})$ if it is metrically regular at $(\bar{x}, \bar{y})$ and the inverse mapping $F^{-1}$ has a single-valued localization around $(\bar{y}, \bar{x}).$

For a single-valued mapping $f \colon X \to Y$ and a point $\bar{x} \in X$ we set
\begin{eqnarray*}
\mathrm{lip} f(\bar{x}) &:=& \mathop{\mathop {\limsup}\limits_{x, \, x' \to \bar{x}} } \limits_{x \ne x'} \frac{\|f(x) - f(x')\|}{\|x - x'\|}.
\end{eqnarray*}
Observe that $\mathrm{lip} f(\bar{x}) < \infty$ if and only if $f$ is Lipschitz continuous around $\bar{x}.$

The following result combines \cite[Theorem~1.5]{Dontchev2003} and \cite[Theorem~4.6]{Dontchev2004} (see also \cite[Theorem~6A.7 and 6A.8]{Dontchev2009}, \cite[Theorem~5.61]{Ioffe2017}, \cite[Theorem~3.4]{Mordukhovich2004}); it provides relationships between the modulus of metric regularity and the radius of (strong) metric regularity.
\begin{DLRTheorem}[radius theorem for (strong) metric regularity] \label{DLR}
For a mapping $F \colon X \rightrightarrows Y$ and any $(\bar{x}, \bar{y}) \in \mathrm{gph} F$ at which $\mathrm{gph} F$ is locally closed, one has
\begin{eqnarray*}
\mathop{\mathop {\inf}\limits_{f \colon X \to Y}} \limits_{f(0) = 0} \{\mathrm{lip} f(\bar{x}) \mid F + f \textrm{ is not metrically regular at $(\bar{x}, \bar{y})$} \} &\ge&
\frac{1}{\mathrm{reg} F (\bar{x}, \bar{y})}.
\end{eqnarray*}
If $F$ is strongly metrically regular at $(\bar{x}, \bar{y}),$ then 
\begin{eqnarray*}
\mathop{\mathop {\inf}\limits_{f \colon X \to Y}} \limits_{f(0) = 0} \{\mathrm{lip} f(\bar{x}) \mid F + f \textrm{ is not strongly metrically regular at $(\bar{x}, \bar{y})$} \} &\ge&
\frac{1}{\mathrm{reg} F (\bar{x}, \bar{y})}.
\end{eqnarray*}
If $X$ and $Y$ are finite-dimensional, then both inequalities become equations. Moreover, the infimum is unchanged if taken with respect to linear mappings of rank $1.$ 
\end{DLRTheorem}

At this point we would like to note that the concept of {\em metric regularity at infinity} of set-valued mappings between finite-dimensional spaces 
is recently introduced in \cite{PHAMTS2023-6} (see also \cite{PHAMTS2023-5, PHAMTS2023-4}).
So it is natural to have a radius theorem for this concept.

\subsection*{Contributions}
Let $F \colon X \rightrightarrows Y$ be a set-valued mapping with closed graph, $\overline{y} \in \mathbb{R}^m$ be a {\em non-proper value} of $F,$ and let $\mathscr{F}$ be the set of  single-valued mappings $f \colon X \to Y$ satisfying the following conditions:
\begin{eqnarray*}
\lim_{\|x\| \to \infty} \|f(x)\| &=& 0 \quad \textrm{ and } \quad  \mathrm{lip} f(\infty) \ < \ \infty.
\end{eqnarray*}
(Notation and definitions will be given in the next sections). The main results of this paper are as follows.
\begin{enumerate}[{\rm (i)}]
\item The following inequality holds true:
\begin{eqnarray*}
\inf_{f \in \mathscr{F}} \{\mathrm{lip} f(\infty) \mid F + f \textrm{ is not metrically regular at $(\infty, \bar{y})$} \} &\ge&
\frac{1}{\mathrm{reg} F (\infty, \bar{y})}.
\end{eqnarray*}

\item If $F$ is strongly metrically regular at $(\infty, \bar{y}),$ then 
\begin{eqnarray*}
\inf_{f \in \mathscr{F}} \{\mathrm{lip} f(\infty) \mid F + f \textrm{ is not strongly metrically regular at $(\infty, \bar{y})$} \} &\ge&
\frac{1}{\mathrm{reg} F (\infty, \bar{y})}.
\end{eqnarray*}

\item If $X$ and $Y$ are finite-dimensional, then both inequalities become equations; moreover, the infimum is unchanged if restricted to mappings $f \in \mathscr{F}$ of rank one Lipschitz.
\end{enumerate}

The main tools of our analysis involve normal cones and coderivatives {\em at infinity}, which are recently introduced in \cite{PHAMTS2023-6}. In particular, the proof of (iii) is based on 
a {\em version at infinity} of the Mordukhovich criterion which states that
\begin{eqnarray*}
\inf\,\{\|x^*\| \mid x^* \in D^*F (\infty, \overline{y})(y^*) \ \textrm{ with } y^* \in Y^* \textrm{ and }  \|y^*\| = 1 \}  &=& 
\frac{1}{\mathrm{reg} F (\infty, \bar{y})},
\end{eqnarray*}
where $D^*F (\infty, \overline{y}) \colon Y^* \rightrightarrows X^*$ denotes the coderivative mapping of $F$ at $(\infty, \overline{y}).$

\medskip

The rest of the paper is organized as follows. Notation and definitions from variational analysis are recalled in Section~\ref{Section2}. The results and their proofs are presented in Sections~\ref{Section3}~and~\ref{Section4}.
Conclusions are given in Section~\ref{Section5}.

\section{Preliminaries} \label{Section2}

\subsection{Notation and definitions}
Throughout the paper we use standard notation, with special symbols introduced where they are defined. 
Unless otherwise stated, all spaces in question are Banach whose norms are always denoted by $\|\cdot\|$ and all finite-dimensional spaces are supposed to be equipped with the usual scalar product $\langle \cdot, \cdot \rangle$ and the corresponding Euclidean norm.

Let $X$ be a Banach space with its dual space $X^*.$ 
We denote by $\mathbb{B}_r(x)$ the open ball centered at $x$ with radius $r;$  when ${x}$ is the origin of $X$ we write $\mathbb{B}_{r}$ instead of $\mathbb{B}_{r}({x}),$ and when $r = 1$ we write  $\mathbb{B}$ instead of $\mathbb{B}_{1}.$

For a nonempty set $\Omega \subset X,$ the closure, interior, convex hull and conic hull of $\Omega$ are denoted, respectively, by $\mathrm{cl}\, {\Omega},$ $\mathrm{int}\, {\Omega},$ $ \mathrm{co}\, \Omega,$ and $\mathrm{cone}\, \Omega.$

The distance in $X$ between a point $x$ and a set $C$ will be denoted by $\mathrm{dist}(x, C);$ thus, $\mathrm{dist}(x, C) := \inf\{\|x - x'\| \mid x' \in C\}.$ A {\em neighborhood of the infinity} is defined as the set $X \setminus {\mathbb{B}}_r$ for some $r > 0.$
We will adopt the convention that $\frac{1}{0} = \infty,$ $\frac{1}{\infty} = 0,$ $\inf \emptyset = \infty$ and $\sup \emptyset = -\infty;$ the notation $x  \to \infty $ means that $\|x\| \to \infty.$

Let $F \colon X \rightrightarrows Y$ be a set-valued mapping between Banach spaces. The {\em domain} and {\em graph} of $F$ are taken to be the sets
\begin{eqnarray*}
\mathrm{dom} F &:=& \{x\in X \mid F(x)\not= \emptyset\}, \\
\mathrm{gph} F&:=& \{(x,y)\in X \times Y \mid y \in F(x)\}.
\end{eqnarray*}
The {\em Painlev\'e--Kuratowski outer limit} of $F$ is given by
\begin{eqnarray*}
\Limsup_{x' \to {x}} F(x') &:=& \{y \in Y \mid \exists x_k \to {x}, \exists y_k \in F(x_k), y_k \to y\}.
\end{eqnarray*}
The inverse mapping $F^{-1} \colon Y \rightrightarrows X$ is defined by $F^{-1}(y) := \{x \in X \mid y \in F(x)\}.$
The mapping $F$ is called {\em positively homogeneous} if $0 \in F(0)$ and $t F(x) \subset F(t x)$ for all $x \in X$ and $t > 0;$
it is {\em sublinear} when, in addition, $F(x + x')  \supset F(x) + F(x').$
When the mapping $F$ is positively homogeneous, we associate the {\em upper norm}
\begin{eqnarray*}
\|F\|_+ &:=& \sup \{\|y\| \mid y \in F(x), \|x\| \le 1\}.
\end{eqnarray*}

\subsection{Normal cones and coderivatives}
Here we recall definitions of normal cones to sets and coderivatives of set-valued mappings, which can be found in~\cite{Mordukhovich2006, Mordukhovich2018, Rockafellar1998}.

In the rest of this section we will assume that $X$ and $Y$ are finite-dimensional spaces. 

\begin{definition}[normal cones]{\rm 
Consider a set $\Omega\subset X$ and a point ${x} \in \Omega.$
\begin{enumerate}
\item[(i)]  The {\em regular normal cone} (known also as the {\em prenormal} or {\em Fr\'echet normal cone}) $\widehat{N}_{\Omega}({x})$ to $\Omega$ at ${x}$ consists of all vectors $x^* \in X^*$ satisfying
\begin{eqnarray*}
\langle x^* , x' - {x} \rangle & \leqslant & o(\|x' -  {x}\|) \quad \textrm{ as } \quad x' \to {x} \quad \textrm{ with } \quad x' \in \Omega.
\end{eqnarray*}

\item[(ii)] 
The {\em limiting normal cone} (known also as the {\em basic} or {\em Mordukhovich normal cone}) $N_{\Omega} ({x})$ to $\Omega$ at ${x}$ consists of all vectors $x^*  \in X^*$ such that there are sequences $x_k \to {x}$ with $x_k \in \Omega$ and $x^*_k \rightarrow x^*$ with $x^*_k \in \widehat N_{\Omega}(x_k),$ or in other words, 
\begin{eqnarray*}
{N}_\Omega({x}) & := &  \Limsup_{x' \xrightarrow{\Omega} {x}}\widehat{N}_\Omega({x}'),
\end{eqnarray*}
where $x' \xrightarrow{\Omega} {x}$ means that $x' \rightarrow {x} $ with $x' \in \Omega.$ 
\end{enumerate}

If $x \not \in \Omega,$ we put $\widehat{N}_{\Omega}({x}) := \emptyset$ and ${N}_\Omega({x}) := \emptyset.$
}\end{definition}

\begin{definition}[coderivative]{\rm 
Consider a set-valued mapping $F \colon X \rightrightarrows Y$ and a point $(\overline{x}, \overline{y}) \in \mathrm{gph} F.$ Assume that $\mathrm{gph} F$ is locally closed at $(\overline{x}, \overline{y}).$ {\em The (basic) coderivative} of $F$ at $(\overline{x}, \overline{y})$ is the set-valued mapping $D^*F(\overline{x}, \overline{y}) \colon Y^* \rightrightarrows X^*$ defined by
\begin{eqnarray*}
D^*F(\overline{x}, \overline{y})(y^*) &=& \{x^* \in X^* \mid (x^*, -y^*)\in N_{\mathrm{gph} F} (\overline{x}, \overline{y}) \} \quad \textrm{ for all } \quad y^* \in Y^*.
\end{eqnarray*}
}\end{definition}

\subsection{Normal cones and coderivatives at infinity}
We also recall definitions of normal cones at infinity to sets and coderivatives at infinity of set-valued mappings; see \cite{PHAMTS2023-6}.

\begin{definition}[normal cone at infinity]{\rm 
Consider a set $\Omega \subset X \times Y$ and a point $\bar{y} \in Y.$ The {\em normal cone to $\Omega$ at $(\infty, \overline{y})$} is defined by
\begin{eqnarray*}
N_{\Omega}(\infty, \overline{y}) &:=& \Limsup_{(x, y) \xrightarrow{\Omega} (\infty, \overline{y})} \widehat{N}_{\Omega}(x, y),
\end{eqnarray*}
where ${(x, y) \xrightarrow{\Omega} (\infty, \overline{y})} $ means that $\|x\| \rightarrow \infty$ and $y \rightarrow \overline{y}$ with $(x, y) \in \Omega.$ 
}\end{definition}

Let $F\colon X \rightrightarrows Y$ be a set-valued mapping with closed graph. We will associate with $F$ the {\em Jelonek set}
\begin{eqnarray*}
J(F) &:=& \{y \in Y\mid \exists (x_k, y_k)  \xrightarrow{\mathrm{gph} F} (\infty, y) \}.
\end{eqnarray*}

\begin{definition}[coderivative at infinity] {\rm 
Let $\overline{y} \in J(F).$ {\em The coderivative} of $F$ at $(\infty, \overline{y})$ is the set-valued mapping $D^*F(\infty, \overline{y}) \colon Y^* \rightrightarrows X^*$ defined by
\begin{eqnarray*}
D^*F(\infty, \overline{y})(y^*) &:=& \{x^* \in X^* \mid (x^*, -y^*)\in N_{\mathrm{gph} F}(\infty, \overline{y})\} \quad \textrm{ for } \ y^* \in Y^*.
\end{eqnarray*}
}\end{definition}

It is easy to see that the coderivative of $F$ at $(\infty, \overline{y})$ is a positively homogeneous mapping with closed graph.
Moreover, we have the following property.

\begin{proposition}\label{Prop25}
Let $\overline{y} \in J(F).$ For all ${y}^* \in Y^*$ one has
\begin{eqnarray*}
D^*F(\infty, \overline{y})({y}^*) &=& \mathop{\mathop {\Limsup}\limits_{(x, y)\xrightarrow{\mathrm{gph} F}(\infty, \overline{y})} }
\limits_{z^*  \to {y}^*} {D^*F(x, y)(z^*)}.
\end{eqnarray*}
\end{proposition}
\begin{proof}
See \cite[Proposition~3.9]{PHAMTS2023-6}.
\end{proof}

\section{Metric regularity at infinity} \label{Section3}

This section concerns the perturbation distance to the failure of metric regularity at infinity. We begin with the following.

\begin{definition}\label{DNSR} {\rm
Let $X$ and $Y$ be Banach spaces, $F \colon X \rightrightarrows Y$ be a set-valued mapping, and $\bar{y} \in J(F).$ 
The mapping $F$ is said to be {\em metrically regular} at $(\infty, \bar{y})$ if there are constants $\kappa > 0, \gamma > 0,$ 
a neighborhood $U$ of the infinity in $X,$ and a neighborhood $V$ of $\overline{y}$ in $Y$ such that for all $(x, y) \in U \times V$ with $\mathrm{dist}\big(y, F(x)\big)  < \gamma,$ we have
\begin{eqnarray*}
\mathrm{dist}\big(x, F^{-1}(y) \big) & \leqslant & \kappa\, \mathrm{dist}\big(y, F(x) \big).
\end{eqnarray*} 
The infimum of the set of values $\kappa$ for which this holds is the {\em modulus of metric regularity} at $(\infty, \bar{y})$ and denoted by $\mathrm{reg} F (\infty, \bar{y}).$ The case $\mathrm{reg} F (\infty, \bar{y}) = \infty$ corresponds to the absence of metric regularity at $(\infty, \bar{y}).$
}\end{definition}

Let $X$ and $Y$ be Banach spaces. For a single-valued mapping $f \colon X \to Y$ we set
\begin{eqnarray*}
\mathrm{lip} f(\infty) &:=& \limsup_{\|x\| \to \infty, \, \|x'\| \to \infty, \, x\ne x'} \frac{\|f(x) - f(x')\|}{\|x - x'\|}.
\end{eqnarray*}
Observe that $\mathrm{lip} f(\infty) < \infty$ if and only if $f$ is {\em Lipschitz continuous at infinity} in the sense that there exist
a constant $\lambda > 0$ and a neighborhood at infinity $U$ in $X$ such that
\begin{eqnarray*}
\|f(x) - f(x')\| &\leqslant & \lambda\|x - x'\| \quad \textrm{ for all } \quad x, x'  \in U.
\end{eqnarray*}

Let $\mathscr{F}$ be the set of  single-valued mappings $f \colon X \to Y$ satisfying the following conditions:
\begin{eqnarray*}
\lim_{\|x\| \to \infty} \|f(x)\| &=& 0 \quad \textrm{ and } \quad  \mathrm{lip} f(\infty) \ < \ \infty.
\end{eqnarray*}

Recall from \cite[p.~118]{Ioffe2017} that a mapping $f \colon X \to Y$  between Banach spaces is {\em rank one Lipschitz} if, for any $x \in X,$ there is a neighborhood $U$ of $x$ in $X$ on which $f$ can be represented in the form 
\begin{eqnarray*}
f(u) &=& \xi(u)y \quad \textrm{ for } \quad u \in U,
\end{eqnarray*}
where $\xi \colon U \to \mathbb{R}$ is Lipschitz continuous and $y \in Y.$

A set-valued mapping $F \colon X \rightrightarrows Y$ with $\bar{y} \in J(F)$ can be perturbed by adding to $F$ a single-valued mapping $f \colon X \rightarrow Y$ with $\lim_{\|x\| \to \infty} \|f(x)\| = 0,$ so as to get a mapping $F + f$ still having $\bar{y} \in J(F + f).$ 
It is natural to measure, with respect to the pair $(\infty, \bar{y}),$ where metric regularity holds, how far $F$ can be perturbed before metric regularity may be lost; equivalently, how big $f$ can be before $F + f$ fails to be metrically regular at $(\infty, \bar{y}).$
This question is answered by the following result.

\begin{theorem}[radius theorem for metric regularity at infinity] \label{MainTheorem1}
Let $X$ and $Y$ be Banach spaces, $F \colon X \rightrightarrows Y$ be a set-valued mapping with closed graph, and $\bar{y} \in J(F).$ Then 
\begin{eqnarray}\label{PT10}
\inf_{f \in \mathscr{F}} \{\mathrm{lip} f(\infty) \mid F + f \textrm{ is not metrically regular at $(\infty, \bar{y})$} \} &\ge&
\frac{1}{\mathrm{reg} F (\infty, \bar{y})}.
\end{eqnarray}
When $X$ and $Y$ are finite-dimensional, the inequality becomes an equation; moreover, the infimum is unchanged if it restricted to mappings $f \in \mathscr{F}$ of rank one Lipschitz.
\end{theorem}

The proof of the theorem will be carried out in several propositions. The first one can be viewed as a version at infinity of the extended Lyusternik--Graves theorem; it provides the effect of Lipschitz perturbations on the modulus of metric regularity at infinity.

\begin{proposition}[extended Lyusternik--Graves theorem at infinity]\label{MD32}
Let $X$ and $Y$ be Banach spaces, $F \colon X \rightrightarrows Y$ be a set-valued mapping with closed graph, and $\bar{y} \in J(F).$ If $\mathrm{reg} F(\infty, \bar{y}) < \kappa < \infty$ and if $f \colon X \to Y$ is a single-valued mapping such that 
\begin{eqnarray*}
\lim_{\|x\| \to \infty} \|f(x)\| = 0 \quad \textrm{ and } \quad \mathrm{lip} f(\infty)  < \lambda < \kappa^{-1},
\end{eqnarray*}
then
\begin{eqnarray*}
\mathrm{reg} (F + f) (\infty, \bar{y}) &<& \frac{\kappa}{1 - \kappa \lambda}.
\end{eqnarray*}
\end{proposition}

\begin{proof} (cf. \cite{Dontchev2003}).
By assumptions, there are constants $r > 0, R > 0$ and $\gamma > 0$ satisfying the following conditions:
\begin{enumerate}[{\rm ({a}1)}]
\item For $x \in X \setminus \mathbb{B}_R$ and $y \in \mathbb{B}_r(\bar{y})$ with $\mathrm{dist}\big(y, F(x)\big) < \gamma,$ we have
\begin{eqnarray*}
\mathrm{dist}\big(x, F^{-1}(y)\big) &\leqslant& \kappa\, \mathrm{dist}\big(y, F(x) \big).
\end{eqnarray*}

\item For $x, x' \in X \setminus \mathbb{B}_R,$ we have
\begin{eqnarray*}
\|f(x)\| \ < \ \frac{r}{2} \quad \textrm{ and } \quad \|f(x) - f(x')\| \ \leqslant \ \lambda\|x - x'\|.
\end{eqnarray*}
\end{enumerate}

Choose $\epsilon_1, \gamma_1, r_1$ and $R_1$ such that
\begin{eqnarray*}
0 &<& \epsilon_1 \ < \ \min \{1 - \kappa \lambda, \frac{\gamma}{\lambda}\}, \\
0 &<& \gamma_1 \ < \ \min \{\gamma, \frac{\gamma - \epsilon_1 \lambda}{\kappa \lambda}\}, \\
0 &<& r_1 \ < \ \frac{r}{2}, \\
R_1 &>& R + \frac{\kappa \gamma_1 + \epsilon_1}{1 - (\kappa \lambda + \epsilon_1)}.
\end{eqnarray*}

Note that $\bar{y} \in J(F + f).$ Let $x \in X \setminus \mathbb{B}_{R_1}$ and $y \in \mathbb{B}_{r_1}(\bar{y})$ with $\mathrm{dist}\big(y, (F + f)(x)\big) < \gamma_1.$ We will show that 
\begin{eqnarray} \label{Eqn*}
\mathrm{dist}\big(x, (F + f)^{-1}(y)\big) &\leqslant& \frac{\kappa}{1 - \kappa \lambda}\, \mathrm{dist}\big(y, (F + f)(x) \big).
\end{eqnarray}

Indeed, we have
\begin{eqnarray*}
\|y - f(x) - \bar{y}\| &\leqslant& \|y - \bar{y}\|  + \|f(x)\| \ < \ r_1 + \frac{r}{2} \ < \ r,
\end{eqnarray*}
and
\begin{eqnarray*}
\mathrm{dist}\big(y - f(x), F(x)\big) &=&  \mathrm{dist}\big(y, (F + f)(x)\big) \ < \ \gamma_1 \ < \ \gamma.
\end{eqnarray*}
The condition~(a1) shows that
\begin{eqnarray*}
\mathrm{dist}\big(x, F^{-1}(y - f(x))\big) &\leqslant& \kappa\, \mathrm{dist}\big(y - f(x), F(x)\big) \ = \ \kappa\, \mathrm{dist}\big(y, (F + f)(x)\big) \ < \ \kappa \gamma_1 \ < \ \infty,
\end{eqnarray*}
and therefore $F^{-1}(y - f(x)) \ne \emptyset.$ Fix $\epsilon$ such that
\begin{eqnarray*}
0 &<& \epsilon \ < \  \epsilon_1.
\end{eqnarray*}
By definition, there exists $z^1 \in F^{-1}(y - f(x))$ such that
\begin{eqnarray*}
\|z^1 - x\| &<& \mathrm{dist}\big(x, F^{-1}(y - f(x))\big) + \epsilon.
\end{eqnarray*}
If $z^1 = x,$ then $x \in F^{-1}(y - f(x)),$ which is the same as $x \in (F + f)^{-1}(y).$ Then \eqref{Eqn*} holds trivially, since its left side is $0.$

Let $z^1 \ne x.$ In this case, we deduce from the condition~(a1) that
\begin{eqnarray*}
\|z^1 - x\| 
&<& \mathrm{dist}\big(x, F^{-1}(y - f(x))\big) + \epsilon \\
&\leqslant& \kappa\, \mathrm{dist}\big(y - f(x), F(x)\big) + \epsilon \\
&=& \kappa\, \mathrm{dist}\big(y, (F + f)(x)\big) + \epsilon \\
&<& \kappa \gamma_1 + \epsilon.
\end{eqnarray*}
Hence
\begin{eqnarray*}
\|z^1\| &\ge& \|x\| - \|z^1 - x\|  \ > \ R_1 - (\kappa \gamma_1 + \epsilon) \ > \ R.
\end{eqnarray*}

By induction, we construct a sequence of points $z^k \in X,$ with $z^0 := x,$ such that, for $k = 0, 1, 2, \ldots,$
\begin{eqnarray}\label{Eqn10}
\|z^k\| &>& R, \quad z^{k + 1} \in F^{-1}(y - f(z^k)), \quad \textrm{ and } \quad 
\|z^{k + 1} - z^k\|  \ \leqslant \ (\kappa \lambda + \epsilon)^k \|z^1 - x\| .
\end{eqnarray}
We already found $z^1$ which gives us \eqref{Eqn10} for $k = 0.$ Suppose that for some integer $n \ge 1$ we have generated $z^1, \ldots, z^n$ satisfying~\eqref{Eqn10}. If $z^n = z^{n - 1},$ then $z^n \in F^{-1}(y - f(z^n))$ and so $z^n \in (F + f)^{-1}(y).$ In this case, we have
\begin{eqnarray*}
\mathrm{dist}\big(x, (F + f)^{-1}(y)\big)
&\leqslant& \|z^n - x\| \ \leqslant \ \sum_{i = 0}^{n - 1} \|z^{i + 1} - z^i\| \\
&\leqslant& \sum_{i = 0}^{n - 1} (\kappa \lambda + \epsilon)^i \|z^1 - x \|  \ < \ \frac{1}{1 - (\kappa \lambda + \epsilon)} \|z^1 - x\| \\
& \leqslant& \frac{\kappa}{1 - (\kappa \lambda + \epsilon)} \left ( \mathrm{dist}\big(y, (F + f)(x)\big) + \frac{\epsilon}{\kappa} \right).
\end{eqnarray*}
Since the left side of this inequality does not depend on $\epsilon$ on the right, we get \eqref{Eqn*} by letting $\epsilon$ go to $0.$

Assume $z^n \ne z^{n - 1}.$ We first show that $\|z^i\| > R$ for all $i = 1, 2, \ldots, n.$ Indeed, for such an $i$ we have 
\begin{eqnarray*}
\|z^i - x\| &\leqslant& \sum_{j = 0}^{i - 1} \|z^{j + 1} - z^j\| \ \leqslant  \ \sum_{j = 0}^{i - 1}(\kappa \lambda + \epsilon)^j \|z^1 - x\| \\
&<& \frac{1}{1 - (\kappa \lambda + \epsilon)} \|z^1 - x\| \\
&<& \frac{\kappa \gamma_1 + \epsilon}{1 - (\kappa \lambda + \epsilon)},\\
\end{eqnarray*}
which implies that
\begin{eqnarray*}
\|z^i\| &\ge& \|x\| - \|z^{i} - x\|  \ > \ R_1 - \frac{\kappa \gamma_1 + \epsilon}{1 - (\kappa \lambda + \epsilon)} \ > \ R.
\end{eqnarray*}
Thus, $\|z^i\| > R$ for $i = 1, \ldots, n.$ This, together with the condition~(a2), implies that
\begin{eqnarray*}
\|y - f(z^n) - \bar{y}\| &\leqslant& \|y - \bar{y}\| + \|f(z^n)\| \ < \ r_1 + \frac{r}{2} \ < \ r,
\end{eqnarray*}
and
\begin{eqnarray*}
\mathrm{dist}\big(y - f(z^n), F(z^n)\big)  &\leqslant & \|(y - f(z^n)) - (y - f(z^{n - 1}))\| \\
&=& \|f(z^n) - f(z^{n - 1}\| \\
& \leqslant &   \lambda \|z^n - z^{n - 1}\| \\ 
& \leqslant &  \lambda (\kappa \lambda + \epsilon)^{n - 1} \|z^1 - x\| \\
&<& \lambda (\kappa \lambda + \epsilon)^{n - 1}(\kappa \gamma_1 + \epsilon) \\
&<& \lambda (\kappa \gamma_1 + \epsilon) \ <  \ \gamma.
\end{eqnarray*}
The condition~(a1) now ensures that
\begin{eqnarray*}
\mathrm{dist}\big(z^n, F^{-1}(y - f(z^{n}))\big)
&\leqslant& \kappa \mathrm{dist}\big(y - f(z^n), F(z^n)\big) \ < \ \kappa \gamma \ < \ \infty,
\end{eqnarray*}
which yields $F^{-1}(y - f(z^{n})) \ne \emptyset.$ Since $\|z^n - z^{n - 1}\| > 0,$ there exists a point $z^{n + 1} \in F^{-1}(y - f(z^{n}))$ such that
\begin{eqnarray*}
\|z^{n + 1} - z^n\| &\leqslant& \mathrm{dist}\big(z^n, F^{-1}(y - f(z^{n}))\big) + \epsilon \|z^n - z^{n - 1}\| ,
\end{eqnarray*}
and then 
\begin{eqnarray*}
\|z^{n + 1} - z^n\| &\leqslant& \kappa \mathrm{dist}\big(y - f(z^n), F(z^n)\big) + \epsilon \|z^n - z^{n - 1}\|.
\end{eqnarray*}
Since $z^{n} \in F^{-1}(y - f(z^{n - 1}))$ and hence $y - f(z^{n - 1}) \in F(z^n),$ by invoking the induction hypothesis, we obtain
\begin{eqnarray*}
\|z^{n + 1} - z^n\| 
&\leqslant& \kappa \|(y - f(z^n))  - (y - f(z^{n - 1}))\| + \epsilon \|z^n - z^{n - 1}\| \\
&=& \kappa \|f(z^n) - f(z^{n - 1})\| + \epsilon \|z^n - z^{n - 1}\| \\
&\leqslant& \kappa \lambda \|z^n - z^{n - 1}\| + \epsilon \|z^n - z^{n - 1}\| \\
&=& (\kappa \lambda + \epsilon) \|z^n - z^{n - 1}\| \\
&\leqslant& (\kappa \lambda + \epsilon)^n \|z^1 - x\|.
\end{eqnarray*}
The induction is complete, and therefore \eqref{Eqn10} holds for all $k.$

Right after \eqref{Eqn10} we showed that when $z^k = z^{k - 1}$ for some $k$ then \eqref{Eqn*} holds. Suppose now that $z^k \ne z^{k - 1}$ for all $k.$ By virtue of the third condition in \eqref{Eqn10}, we see for any integers $n$ and $m$ with $m < n$ that
\begin{eqnarray*}
\|z^{n} - z^m\| 
&\leqslant& \sum_{k = m}^{n - 1} \|z^{k + 1} - z^{k}\| \\
&\leqslant& \sum_{k = m}^{n - 1} (\kappa \lambda + \epsilon)^k \|z^1 - x \| \\
&<& \frac{(\kappa \lambda + \epsilon)^m}{1 - (\kappa \lambda + \epsilon)} \|z^1 - x \|.
\end{eqnarray*} 
Hence the sequence $\{z^k\}$ satisfies the Cauchy condition, and so is convergent to some $z.$ Since the graph of $F$ is a closed set in $X \times Y,$ $y - f(z) \in F(z),$ that is $z \in (F + f)^{-1}(y).$ Moreover,
\begin{eqnarray*}
\mathrm{dist}\big(x, (F + f)^{-1}(y)\big)
&\leqslant& \|z - x\| \ = \ \lim_{k \to \infty} \|z^k - x\| \\
&\leqslant& \lim_{k \to \infty} \sum_{i = 0}^{k} \|z^{i + 1} - z^{i}\| \\
&\leqslant& \lim_{k \to \infty} \sum_{i = 0}^{k} (\kappa \lambda + \epsilon)^i \|z^{1} - x \| \\
&<& \frac{1}{1 - (\kappa \lambda + \epsilon)} \|z^{1} - x \| \\
&\leqslant& \frac{1}{1 - (\kappa \lambda + \epsilon)} \left( \kappa \mathrm{dist}\big(y, (F + f)(x)\big) + \epsilon\right) \\
&=& \frac{\kappa}{1 - (\kappa \lambda + \epsilon)} \left( \mathrm{dist}\big(y, (F + f)(x)\big) + \frac{\epsilon}{\kappa}\right).
\end{eqnarray*}
Taking the limit as $\epsilon \to 0,$ we obtain \eqref{Eqn*}, and the proof is complete.
\end{proof}

For a set-valued mapping $F \colon X \rightrightarrows  Y$ between finite-dimensional spaces and a point $\bar{y} \in J(F),$ we define a nonnegative quantity closely related to the regularity modulus of $F$ at infinity:
\begin{eqnarray*}
\mathrm{rg}^{+} F(\infty, \bar{y}) &:=& \inf\,\{\|x^*\| \mid x^* \in D^*F (\infty, \overline{y})(y^*) \ \textrm{ with } y^* \in Y^* \textrm{ and }  \|y^*\| = 1 \} .
\end{eqnarray*}
Note that (see, for example, \cite[Proposition~5.4]{Ioffe2017}) the reciprocal of this quantity equals the upper norm of the mapping $D^*F (\infty, \overline{y})^{-1}.$

The validity of the following result, which can be viewed as a version at infinity of the Mordukhovich criterion (see \cite{Mordukhovich1993}), itself will be crucial to our argument for establishing the second statement of Theorem~\ref{MainTheorem1}.

\begin{proposition}[Mordukhovich criterion at infinity] \label{MD33}
Let $X$ and $Y$ be finite-dimensional spaces, $F \colon X \rightrightarrows Y$ be a set-valued mapping with closed graph, and $\bar{y} \in J(F).$ Then
\begin{eqnarray*}
\mathrm{rg}^{+} F(\infty, \bar{y})  &=& \frac{1}{\mathrm{reg} F (\infty, \bar{y})}.
\end{eqnarray*} 
\end{proposition}

\begin{proof}
We first show that 
\begin{eqnarray*}
\mathrm{rg}^+ F(\infty, \bar{y}) & \geqslant & \frac{1}{\mathrm{reg}F (\infty, \bar{y})}.
\end{eqnarray*} 
There is nothing to prove when $\mathrm{reg}F (\infty, \bar{y}) = \infty,$ so assume $\mathrm{reg}F (\infty, \bar{y}) < \infty$ and take any $\kappa > \mathrm{reg}F (\infty, \bar{y}).$ There are a constant $\gamma > 0,$ a neighborhood $U$ of the infinity in $X$ and a neighborhood $V$ of $\overline{y}$ in $Y$ such that
\begin{eqnarray} \label{3PT3}
\mathrm{dist}(x, F^{-1}(y)) & \leqslant & \kappa\, \mathrm{dist}(y, F(x))
\end{eqnarray} 
for all $x \in U$ and $y \in V$ with $\mathrm{dist}(y, F(x))  < \gamma.$ 
Choose a positive constant $\epsilon < 2 \gamma$ such that $\mathbb{B}_{\epsilon}(\overline{y}) \subset V.$  

By definition, there exist sequences $(x_k, y_k) \in \mathrm{gph} F$ and $(x^*_k, -y^*_k)\in \widehat{N}_{\mathrm{gph} F}(x_k, y_k)$ such that
\begin{eqnarray*}
 (x_k, y_k) \to (\infty, \overline{y}), \quad \|x^*_k\| \to \mathrm{rg}^+F(\infty, \bar{y}), \quad  \textrm{ and } \quad \|y^*_k\| \to 1 \quad \textrm{ as } \quad k \to \infty.
\end{eqnarray*}
In particular, $x_k \in U$ and $y_k \in \mathbb{B}_{\epsilon/2}(\overline{y})$ for all $k$ sufficiently large. 

By the definition of  $\widehat{N}_{\mathrm{gph} F}(x_k, y_k),$ for each $k > 0$ we can find $\delta_k > 0$ such that 
\begin{eqnarray}\label{3PT4}
\langle x^*_k, x - x_k\rangle - \langle y^*_k, y - y_k\rangle &  \leqslant  & \frac{1}{k} \big(\|x - x_k\| + \|y - y_k\|\big)
\end{eqnarray}   
for all $(x, y) \in \mathrm{gph} F\cap \big (\mathbb{B}_{\delta_k}(x_k) \times \mathbb{B}_{\delta_k}(y_k)\big).$ 

Fix $k$ sufficiently large so that  $x_k \in U$ and $y_k \in \mathbb{B}_{\epsilon/2}(\overline{y}).$ 
Let $y := y_k - \frac{r}{\kappa} \frac{y^*_k}{\|y^*\|}$ with $0 < r < \min\{\delta_k, \kappa \delta_k, \dfrac{\kappa \epsilon}{2} \}.$ Then $y_k \in  \mathbb{B}_{\delta_k}(y_k)$ and 
\begin{eqnarray*}
\|y - \bar{y} \| &\leqslant& \|y_k - \bar{y} \|  + \frac{r}{\kappa} \ < \ \frac{\epsilon}{2} + \frac{\epsilon}{2}  \ = \ \epsilon.
\end{eqnarray*}
In particular, $y \in V.$ Moreover, we have
\begin{eqnarray*}
\mathrm{dist}(y, F(x_k)) & \leqslant & \|y - y_k\| \ = \ \frac{r}{\kappa} \ < \ \frac{\epsilon}{2} \ < \ \gamma.
\end{eqnarray*} 
Therefore, the inequality~\eqref{3PT3} and the closedness of the mapping $F$ ensure the existence of $x \in F^{-1}(y)$ such that 
\begin{eqnarray*}
\|x - x_k\| & \leqslant & \kappa\, \mathrm{dist}(y, F(x_k)) \ \leqslant \ \kappa \|y - y_k\| \ = \ r \ < \ \delta_k.
\end{eqnarray*}
This, together with the inequality~\eqref{3PT4}, implies easily that
\begin{eqnarray*}
\frac{1}{\kappa} r \|y^*_k\| & \leqslant & 
\|x^*_k\| \|x - x_k\| + \frac{1}{k} \big(\|x - x_k\| + \|y - y_k\|\big) \ \leqslant \  \|x^*_k\| r  + \frac{1}{k}( r + \frac{1}{\kappa} r) ,
\end{eqnarray*} 
which yields
\begin{eqnarray*}
\frac{1}{\kappa} \|y^*_k\| & \leqslant & \|x^*_k\|  + \frac{1}{k}( 1 + \frac{1}{\kappa}) .
\end{eqnarray*} 
Taking the limit as $k\to \infty,$ we obtain $\frac{1}{\kappa}  \leqslant \mathrm{rg}^+ F(\infty, \bar{y}).$ 
Since $\kappa > \mathrm{reg}F (\infty, \bar{y})$ was arbitrary, we conclude that
\begin{eqnarray*}
\frac{1}{\mathrm{reg}F (\infty, \bar{y})}  & \leqslant & \mathrm{rg}^+ F(\infty, \bar{y}).
\end{eqnarray*} 

To verify the opposite inequality, take any $\mu \in (0, \mathrm{rg}^+ F(\infty, \bar{y})).$ We will show that $\mu \leqslant \frac{1}{\mathrm{reg} F(\infty, \bar{y})},$ or equivalently, there are a constant $\gamma > 0,$ a neighborhood $U$ of the infinity in $X,$ and a neighborhood $V$ of $\overline{y}$ in $Y$ such that for all $x \in U$ and $y \in V$ with $\mathrm{dist}\big(y, F(x)\big)  < \gamma,$ we have
\begin{eqnarray*}
\mathrm{dist}\big(x, F^{-1}(y) \big) & \leqslant & \frac{1}{\mu}\, \mathrm{dist}\big(y, F(x) \big).
\end{eqnarray*} 

Indeed, if this is not the case, we can find a sequence $(x_k, y_k) \to (\infty, \overline{y})$ with $\mathrm{dist}\big(y_k, F(x_k)\big) \to 0$ such that
\begin{eqnarray} \label{PTM1}
\mathrm{dist}\big(x_k, F^{-1}(y_k) \big) & > & \frac{1}{\mu}\, \mathrm{dist}\big(y_k, F(x_k) \big).
\end{eqnarray}
Since the mapping $F$ has closed graph, for each $k$, there exists a point $z_k \in F(x_k)$ such that $r_k := \|y_k - z_k\| = \mathrm{dist}\big(y_k, F(x_k) \big).$ Certainly $r_k > 0$ and $r_k \to 0.$

Fix $k$ and define the function $\varphi \colon \mathbb{R}^n\times\mathbb{R}^m \to \overline{\mathbb{R}}$ by
\begin{eqnarray*}
\varphi (x, y) & := & \|y - y_k\|+\delta_{\mathrm{gph} F}(x, y),
\end{eqnarray*}
which is lower semi-continuous as the graph of $F$ is closed. Clearly $\varphi$ is nonnegative and 
$\varphi (x_k, z_k) = r_k \to 0$ as $k \to \infty,$ and so $\inf_{(x, y) \in \mathbb{R}^n\times\mathbb{R}^m} \varphi (x, y) = 0.$
Applying the Ekeland variational principle (see \cite[Theorem~1.1]{Ekeland1974}) to the function $\varphi$ with $\epsilon := r_k, \lambda := \frac{r_k}{\mu},$ the initial point $(x_k, z_k),$ and the metric space 
$(\mathbb{R}^n\times\mathbb{R}^m, d),$ where 
\begin{eqnarray*}
d\big((x, y), (x', y') \big) &:=& \|x - x'\| + \sqrt{r_k} \|y - y'\| \quad \textrm{ for } \quad (x, y), (x', y') \in \mathbb{R}^n\times\mathbb{R}^m,
\end{eqnarray*}
we find a pair $(\overline{x}_k, \overline{z}_k) \in \mathrm{gph}\, F$ such that 
\begin{eqnarray} \label{PTM2}
\|x_k - \overline{x}_k\| + \sqrt{r_k} \|z_k - \overline z_k\| & \leqslant & \frac{r_k}{\mu}
\end{eqnarray}
and $(\overline{x}_k, \overline z_k)$ is a global minimizer of the lower semi-continuous function 
\begin{eqnarray*}
\psi \colon \mathbb{R}^n \times \mathbb{R}^m \to\overline{\mathbb{R}}, && (x, y) \mapsto \|y - y_k\| + {\mu} \big(\|x - \overline{x}_k\| + \sqrt{r_k} \|y - \overline{z}_k\| \big) + \delta_{\mathrm{gph} F}(x, y).
\end{eqnarray*}
According to the Fermat rule, $(0, 0) \in \partial \psi (\overline{x}_k, \overline{z}_k).$ 
If $\overline{z}_k = y_k,$ then $\overline{x}_k \in F^{-1}(y_k),$ and by the inequalities~\eqref{PTM1}~and~\eqref{PTM2}, we get
\begin{eqnarray*}
\frac{r_k}{\mu} + \sqrt{r_k} \|z_k - \overline z_k\|  \ < \ \|x_k - \overline{x}_k\| + \sqrt{r_k} \|z_k - \overline z_k\| & \leqslant & \frac{r_k}{\mu},
\end{eqnarray*}
a contradiction. Therefore, $\overline{z}_k \ne y_k.$ Using some basic properties of subdifferentials (see \cite{Mordukhovich2018, Rockafellar1998}), we obtain
\begin{eqnarray*}
(0, 0)  &\in& \{0\} \times \frac{\overline{z}_k - y_k}{\|\overline{z}_k - y_k\|}  + {\mu} \big( \mathbb{B} \times \sqrt{r_k} \mathbb{B} \big) + {N}_{\mathrm{gph} F}(\overline{x}_k, \overline{z}_k).
\end{eqnarray*}
Hence there is a vector $(x^*_k, y^*_k) \in X^* \times Y^*$ satisfying
\begin{eqnarray*}
\|x^*_k\| \leqslant {\mu},  \quad  1 - \mu \sqrt{r_k}  \leqslant \|y^*_k\| \leqslant 1 + \mu \sqrt{r_k} \quad \textrm{ and } \quad  x^*_k \in D^*F(\overline{x}_k, \overline{z}_k)(y^*_k).
\end{eqnarray*}
Passing to subsequences if necessary, we may assume that $(x^*_k, y^*_k)\to (x^*, y^*).$ Certainly $\|x^*\| \leqslant \mu$ and $\|y^*\| = 1.$ Observe that $(\overline{x}_k,  \overline{z}_k) \to (\infty, \overline{y}).$ Hence by Proposition~\ref{Prop25} $x^* \in D^* F(\infty, \overline y)(y^*),$ which contradicts the choice of $\mu.$ Therefore, 
\begin{eqnarray*}
\frac{1}{\mathrm{reg}F (\infty, \bar{y})}  & \geqslant & \mathrm{rg}^+ F(\infty, \bar{y}).
\end{eqnarray*} 
The proof is complete.
\end{proof}

Another ingredient of our proof of Theorem~\ref{MainTheorem1} will be the following result; see also \cite{Gfrerer2023-2, Gfrerer2023-1}.

\begin{proposition} \label{MD34}
Let $X$ and $Y$ be finite-dimensional spaces, $F \colon X \rightrightarrows Y$ be a set-valued mapping with closed graph, and $\bar{y} \in J(F).$ If $\mathrm{rg}^{+} F(\infty, \bar{y}) < \infty,$ then there exists a single-valued  mapping $f \colon X \to Y$ of rank one Lipschitz such that
\begin{eqnarray} \label{PT7}
\lim_{\|x\| \to \infty} \|f(x)\| = 0, \quad \mathrm{lip} f(\infty) \leqslant \mathrm{rg}^{+} F(\infty, \bar{y}), \quad \textrm{and} \quad \mathrm{rg}^{+} (F + f)(\infty, \bar{y}) =  0.
\end{eqnarray}
\end{proposition}

\begin{proof}
Since the assertion of the lemma is trivially true with $f \equiv 0$ when $\mathrm{rg}^{+} F(\infty, \bar{y}) = 0,$ we can suppose that $\mathrm{rg}^{+} F(\infty, \bar{y}) > 0.$ By Proposition~\ref{Prop25}, there are sequences $(x_k, y_k) \in \mathrm{gph} F$ with $(x_k, y_k)  \to (\infty, \bar{y}),$  $y_k^* \in Y^*$ with $\|y_k^*\| = 1,$ and $x_k^* \in D^*F(x_k,y_k)(y_k^*)$ such that
\begin{eqnarray*}
\lim_{k \to \infty} \|x_k^*\| &=& \mathrm{rg}^{+} F(\infty, \bar{y}) . 
\end{eqnarray*}
Without loss of generality, we may assume that 
\begin{eqnarray*}
0 \ < \ \inf_{k \in \mathbb{N}} \|x_k^*\|  &\leqslant& \sup_{k \in \mathbb{N}} \|x_k^*\| \ < \ \infty 
\end{eqnarray*}
Let $t_k := \frac{k}{ k + 1}\frac{\mathrm{rg}^{+} F(\infty, \bar{y}) }{\|x_k^*\|}.$ Then $t_k > 0$ and $t_k \to 1$ as $k \to \infty.$

Passing to a subsequence if necessary, we can assume that $\|x_{k} \| + 1 < \|x_{k + 1}\|$ for all $k.$ Then certainly there exists a sequence $\{\rho_k\}$
of positive real numbers such that 
\begin{eqnarray*}
\lim_{k\to \infty} \rho_k = 0 \quad {\rm and} \quad \overline{\mathbb{B}}_{\rho_k}(x_k) \cap \overline{\mathbb{B}}_{\rho_l}(x_l) = \emptyset \quad {\rm for }\quad k\neq l. 
\end{eqnarray*}
For each $k \geqslant 1,$ choose an element $v_k \in Y$ with $\|v_k\| = 1$ satisfying
\begin{eqnarray}\label{eq vk}
\langle y_k^*, v_k \rangle &>& 1-\dfrac{1}{k},
\end{eqnarray}
and define the mapping $f_k \colon X \to Y, x \mapsto f_k(x),$ by
\begin{eqnarray*}
f_k(x) & := & t_k s_k(x) \langle x_k^*, x-x_k\rangle v_k,
\end{eqnarray*}
where
\begin{eqnarray*}
s_k(x)& := & {\rm max} \left\{ 1 - \left(\dfrac{\|x-x_k\|}{\rho_k}\right)^{1+\frac{1}{k}},0 \right\}.
\end{eqnarray*}
Then the mapping $f_k$ is continuous differentiable around $x_k$ with
\begin{eqnarray*}
D^*f_k(x_k)(y_k^*) &=& - t_k \langle y_k^*,v_k \rangle x_k^*.
\end{eqnarray*}
Moreover, from the the proof of \cite[Lemma~3.1]{Gfrerer2023-2} we know that $f_k$ is Lipschitz continuous on $\overline{\mathbb{B}}_{\rho_k}(x_k)$ with modulus $\mathrm{rg}^{+} F(\infty, \bar{y}).$

Next we define the mapping $f \colon X \to Y, x \mapsto f(x),$ by
\begin{eqnarray*}
f(x) &:=& - \sum_{k = 1}^\infty f_k(x).
\end{eqnarray*}
By construction, the mapping $f$ is well-defined, $f(x) = f_k(x)$ for all $x \in \overline{\mathbb{B}}_{\rho_k}(x_k)$
and all $k \geqslant 1,$ and $f(x) = 0$ for all $x \notin \bigcup_{k = 1}^\infty \overline{\mathbb{B}}_{\rho_k}(x_k).$
We will show that the mapping $f$ has the desired properties.

We first show that $f$ is rank one Lipschitz. Indeed, let $x \in X.$ If $x \in \overline{\mathbb{B}}_{\rho_k}(x_k)$ for some $k \geqslant 1$ then there exists $\widetilde{\rho}_k > \rho_k$ such that $f(u) = - t_k s_k(u) \langle x_k^*, u - x_k\rangle v_k$ for all $u \in {\mathbb{B}}_{\widetilde{\rho}_k}(x_k).$ Note that ${\mathbb{B}}_{\tilde{\rho}_k}(x_k)$ is a neighborhood of $x.$ If $x \not \in \cup_{k = 1}^\infty \overline{\mathbb{B}}_{\rho_k}(x_k)$ then $f(u) = 0$ for all $u$ in a sufficiently small neighborhood of $x.$ Hence $f$ is 
rank one Lipschitz.

For the first equation in \eqref{PT7}, suppose $\|z_l\| \to \infty$ with $z_l \in X.$ Then, for each integer $l,$ either $z_l \notin \bigcup_{k \geqslant 1 }\overline{\mathbb{B}}_{\rho_k}(x_k)$ or $z_l \in \mathbb{B}_{\rho_k}(x_k)$ for some $k$ (depending on $l$). 
 In the first case one has $f(z_l) = 0,$ while in the second case the conclusion is that
\begin{eqnarray*}
\|f(z_l)\| &=& \|f_k(z_l)\| \leqslant \ t_k \|x_k^*\| \|z_l - x_k\| \ \leqslant \ t_k \rho_k \|x_k^*\|.
\end{eqnarray*}
It follows that $\|f(z_l)\| \to 0$ as $l \to \infty.$ The arbitrary choice of the sequence $z_l$ ensures that 
\begin{eqnarray*}
\lim_{\|x\| \to \infty} \|f(x)\| &=& 0,
\end{eqnarray*}
as required.

We next prove that the mapping $f$ is Lipschitz continuous on $X$ with modulus $\gamma :=  \mathrm{rg}^{+} F(\infty, \bar{y}).$ To this end, take any $x, x' \in X$. The four cases should be considered. 

\subsubsection*{Case 1: $x, x' \notin \bigcup_{k = 1}^\infty \overline{\mathbb{B}}_{\rho_k}(x_k)$}
We have $f(x) = f(x') = 0,$ and so $\|f(x) - f(x')\| \leqslant \gamma \|x-x'\|.$
 
\subsubsection*{Case 2: $x \in \mathbb{B}_{\rho_k}(x_k)$ for some $k$ and $x' \notin  \bigcup_{k = 1}^\infty  \overline{\mathbb{B}}_{\rho_k}(x_k)$}

The segment $[x, x'] := \{ tx  + (1 - tx') \mid t \in [0,1]\}$ intersects the boundary of $\mathbb{B}_{\rho_k}(x_k)$ at least one point $z.$
Note that $f(z) = f_k(z) = 0$ and $f(x') = 0.$ Hence
\begin{eqnarray*}
\|f(x) - f(x')\| 
& \leqslant & \|f(x) - f(z)\|+ \|f(z) - f(x')\|  \\
& = & \|f_k(x) - f_k(z)\| \ \leqslant \ \gamma \|x-z\| \ \leqslant \ \gamma \|x-x'\|. 
\end{eqnarray*}

\subsubsection*{Case 3: $x \in \mathbb{B}_{\rho_k}(x_k)$ and $x' \in \mathbb{B}_{\rho_l}(x_l)$  for some $k, l$ with $k \ne l$}
The segment $[x, x']$ intersects the boundary of the ball $\mathbb{B}_{\rho_k}(x_k)$ (resp., 
$\mathbb{B}_{\rho_l}(x_l)$) at least one point $z$ (resp., $z'$). We have $f(z) = f_k(z) = 0$ and $f(z') = f_l(z') = 0.$ Thus,
\begin{eqnarray*}
\|f(x) - f(x')\| 
& \leqslant & \|f(x) - f(z)\|+ \|f(z) - f(z')\|+\|f(z') - f(x')\|  \\
& = & \|f(x) - f(z)\| + \|f(z') - f(x')\| \\
& = & \|f_k(x) - f_k(z)\| + \|f_l(z') - f_l(x')\| \\
& \leqslant & \gamma \|x-z\| + \gamma \|z' - x'\| \ \leqslant \ \gamma \|x - x'\|. 
\end{eqnarray*}

\subsubsection*{Case 4: $x, x' \in \mathbb{B}_{\rho_k}(x_k)$ for some $k$}
Then, $f(x) = f_k(x)$ and $f(x') = f_k(x').$ Consequently, 
\begin{eqnarray*}
\|f(x) - f(x')\|  & = & \|f_k(x) - f_k(x')\| \ \leqslant \ \gamma \|x - x'\|. 
\end{eqnarray*}

Therefore, in all cases, the mapping $f$ is Lipschitz continuous on $X$ with modulus $\gamma.$ The inequality in \eqref{PT7} is proved.

Finally, we prove the last equation in \eqref{PT7}. To this end, take any $k \geqslant 1.$ A direct calculation shows that
the mapping $f$ is continuous differentiable around $x_k$ with
\begin{eqnarray*}
D^*f(x_k)(y_k^*) &=& - t_k \langle y_k^*,v_k \rangle x_k^*.
\end{eqnarray*}
This, together with the sum rule in \cite[Theorem~3.9]{Mordukhovich2018}, implies that
\begin{eqnarray*}
\left(1 - t_k \langle y_k^*, v_k \rangle \right)x_k^* &\in & D^*F(x_k,y_k)(y_k^*) + D^*f(x_k)(y_k^*) \ \subset \ D^*(F + f)(x_k,y_k)(y_k^*).  
\end{eqnarray*}
Taking the limit as $k \to \infty,$ we obtain
\begin{eqnarray*}
\mathrm{rg}^{+} (F + f)(\infty, \bar{y}) & \leqslant & \lim_{k \to \infty} \left( 1- t_k \left(1 - \frac{1}{k}\right)\right)\|x_k^*\| \ = \ 0. 
\end{eqnarray*}
The proof is complete. 
\end{proof}

We are now in a position to prove the main result of this section.

\begin{proof}[Proof of Theorem~\ref{MainTheorem1}]
We first verify the inequality~\eqref{PT10}. By contradiction, suppose that there is a mapping $f \in \mathscr{F}$ such that $\mathrm{lip} f(\infty) < \big( \mathrm{reg} F(\infty, \bar{y}) \big)^{-1}$ and $F + f $ is not metrically regular at $(\infty, \bar{y}).$ Then we can find constants $\lambda > \mathrm{lip} f(\infty) $ and $\kappa > \mathrm{reg} F(\infty, \bar{y}) $ such that $\lambda < \kappa^{-1}.$ According to Proposition~\ref{MD32}, we obtain
\begin{eqnarray*}
\mathrm{reg} (F + f) (\infty, \bar{y}) &<& \frac{\kappa}{1 - \kappa \lambda},
\end{eqnarray*}
which yields $F + f$ is metrically regular at $(\infty, \bar{y}),$ a contradiction. Therefore, the inequality~\eqref{PT10} holds.

We now assume that $\dim X < \infty$ and $\dim Y < \infty.$ Since the second assertion of the theorem is trivially true when $\mathrm{reg}F(\infty, \bar{y})  = 0$ or $\mathrm{reg}F(\infty, \bar{y}) = \infty,$ we may suppose that $0 < \mathrm{reg} F(\infty, \bar{y}) < \infty.$ 
By Proposition~\ref{MD33}, we have
\begin{eqnarray} \label{PT9}
\mathrm{rg}^{+} F(\infty, \bar{y})  &=& \frac{1}{\mathrm{reg} F (\infty, \bar{y})} \ < \ \infty.
\end{eqnarray}

In view of Proposition~\ref{MD34}, there exists a mapping $f \in \mathscr{F}$ of rank one Lipschitz such that 
\begin{eqnarray*}
\mathrm{lip} f(\infty) \leqslant \mathrm{rg}^{+} F(\infty, \bar{y})  \quad \textrm{ and } \quad  \mathrm{rg}^{+} (F + f)(\infty, \bar{y}) = 0.
\end{eqnarray*}
This, together with Proposition~\ref{MD33} again, indicates the absence of metric regularity of $F + f$ at $(\infty, \bar{y}).$ Therefore,
\begin{eqnarray*}
\mathrm{rg}^{+} F(\infty, \bar{y}) 
&\ge& \mathrm{lip} f(\infty) \\
&\ge&  \inf_{} \{\mathrm{lip} g(\infty) \mid F + g \textrm{ is not metrically regular at $(\infty, \bar{y})$} \},
\end{eqnarray*}
where the infimum is taken over all mappings $g \in \mathscr{F}$ of rank one Lipschitz. It follows from~\eqref{PT9} that the inequality~\eqref{PT10} holds as an equation.
\end{proof}

\section{Strong metric regularity at infinity} \label{Section4}

In this section, we study the radius of strong metric regularity at infinity. We begin with the following.
\begin{definition}{\rm
Let $X$ and $Y$ be Banach spaces, $F \colon X \rightrightarrows Y$ be a set-valued mapping, and $\bar{y} \in J(F).$ 
The mapping $F$ is said to be {\em strongly metrically regular at $(\infty, \bar{y})$} if the metric regularity condition in Definition~\ref{DNSR} is satisfied by some $\kappa, \gamma, U$ and $V$ such that, in addition, the graphical localization  of $F^{-1}$ with respect to $U$ and $V$ is nowhere multivalued, i.e., when $y \in \mathbb{B}_r(\bar{y}),$ there is at most one solution $x \in U$ to the generalized equation $y \in F(x).$
}\end{definition}

The next result is a perturbation radius theorem for strong metric regularity at infinity to go along with the one for metric regularity at infinity.

\begin{theorem}[radius theorem for strong metric regularity at infinity] \label{MainTheorem2}
Let $X$ and $Y$ be Banach spaces, $F \colon X \rightrightarrows Y$ be a set-valued mapping with closed graph, and $\bar{y} \in J(F).$ If $F$ is strongly metrically regular at $(\infty, \bar{y})$ then 
\begin{eqnarray*}
\inf_{f \in \mathscr{F}} \{\mathrm{lip} f(\infty) \mid F + f \textrm{ is not strongly metrically regular at $(\infty, \bar{y})$} \} &\ge&
\frac{1}{\mathrm{reg} F (\infty, \bar{y})}.
\end{eqnarray*}
When $X$ and $Y$ are finite-dimensional, the inequality becomes an equation; moreover, the infimum is unchanged if restricted to mappings $f \in \mathscr{F}$ of rank one Lipschitz.
\end{theorem}

For proving the theorem, we will need two propositions. The first one reads as follows.

\begin{proposition} \label{BD37}
Let $X$ and $Y$ be Banach spaces, $F \colon X \rightrightarrows Y$ be a set-valued mapping with closed graph, and $\bar{y} \in J(F).$ 
The following properties are equivalent:
\begin{enumerate}[{\rm (i)}]
\item The mapping $F$ is strongly metrically regular at $(\infty, \bar{y}).$

\item There exist neighborhoods $U$ of the infinity in $X$ and $V$ of $\bar{y}$ in $Y$ such that the mapping $y \mapsto F^{-1}(y) \cap U$ is single-valued on $V$ and moreover Lipschitz continuous.
\end{enumerate}
\end{proposition}

\begin{proof}
(i) $\Rightarrow$ (ii). By assumptions, there are constants $\kappa > 0, r > 0, R > 0$ and $\gamma > 0$ satisfying the following conditions:
\begin{enumerate}[{\rm (b1)}]
\item For $x \in X \setminus \mathbb{B}_R$ and $y \in \mathbb{B}_r(\bar{y})$ with $\mathrm{dist}\big(y, F(x)\big) < \gamma,$ we have
\begin{eqnarray*}
\mathrm{dist}\big(x, F^{-1}(y)\big) &\leqslant& \kappa \mathrm{dist}\big(y, F(x)\big).
\end{eqnarray*}

\item When $y \in \mathbb{B}_r(\bar{y}),$ there is at most one solution $x \in X \setminus \mathbb{B}_R$ to the generalized equation $y \in F(x).$
\end{enumerate}

Fix $r_1$ such that 
\begin{eqnarray*}
0 &<& r_1 \ < \ \frac{1}{2} \min \{r, \gamma\}.
\end{eqnarray*}
We first show that
\begin{eqnarray} \label{LIP}
F^{-1}(y') \cap \big(X \setminus \mathbb{B}_{R}\big) & \subset & F^{-1}(y) + \kappa \|y' - y\|\mathbb{B} \quad \textrm{for all} \quad y, y' \in \mathbb{B}_{r_1}(\bar{y}).
\end{eqnarray} 
To this end, let $y, y' \in \mathbb{B}_{r_1}(\bar{y}).$ We have for all $x \in F^{-1}(y') \cap \big(X \setminus \mathbb{B}_{R}\big),$
\begin{eqnarray*}
\mathrm{dist}\big(y, F(x)\big) &\leqslant& \|y - y'\| \ \leqslant \ \|y - \bar{y}\| + \|y' - \bar{y}\| \ < \ 2r_1 \ < \ \gamma,
\end{eqnarray*} 
which, together with the condition~(b1), implies that
\begin{eqnarray*}
\mathrm{dist}\big(x, F^{-1}(y)\big) &\leqslant& \kappa \mathrm{dist}\big(y, F(x)\big) \ \leqslant \ \kappa \|y - y'\|.
\end{eqnarray*} 
Therefore, \eqref{LIP} holds.

We next show that the mapping $y \mapsto F^{-1}(y) \cap \big(X \setminus \mathbb{B}_{R}\big)$ is single-valued on $\mathbb{B}_{r_1}(\bar{y}).$ Indeed, take any $y \in \mathbb{B}_{r_1}(\bar{y}).$ Since $\bar{y} \in J(F),$ there exists a sequence $(x^k, y^k) \in \mathrm{gph}F$ such that $\|x^k\| \to \infty$ and $\|y^k - \bar{y}\| \to 0$  as $k \to \infty.$ Then for all $k$ sufficiently large, $x^k \in F^{-1}(y^k) \cap \big(X \setminus \mathbb{B}_{R}\big)$
and $y^k \in \mathbb{B}_{r_1}(\bar{y}),$ which, by \eqref{LIP}, yield $x^k \in F^{-1}(y) + \kappa \|y^k - y\|.$ Thus, $F^{-1}(y)  \cap \big(X \setminus \mathbb{B}_{R}\big) \ne \emptyset.$ This, together with the condition~(b2), implies that the mapping $y \mapsto F^{-1}(y) \cap \big(X \setminus \mathbb{B}_{R}\big)$ is single-valued on $\mathbb{B}_{r_1}(\bar{y})$ and moreover Lipschitz continuous with constant $\kappa$ in view of  \eqref{LIP} again.

(ii) $\Rightarrow$ (i). Let $\kappa > 0, R > 0$ and $r > 0$ be such that the mapping $y \mapsto F^{-1}(y) \cap (X \setminus \mathbb{B}_R)$ is single-valued on $\mathbb{B}_r(\bar{y})$ and Lipschitz continuous with constant $\kappa.$ Denote by $\phi(y)$ the unique element of $F^{-1}(y) \cap (X \setminus \mathbb{B}_R),$ when $y \in \mathbb{B}_r(\bar{y}).$ We have
\begin{eqnarray*}
\| \phi(y') - \phi(y)\| &\leqslant&  \kappa \|y' - y\| \quad \textrm{ for all } \quad y, y' \in \mathbb{B}_r(\bar{y}).
\end{eqnarray*}

Fix $\gamma$ such that
\begin{eqnarray*}
0 \ < \ \gamma \ < \frac{r}{2}. 
\end{eqnarray*}
Let $x \in X \setminus \mathbb{B}_{R}$ and $y \in \mathbb{B}_{\frac{r}{2}}(\bar{y})$ with $\mathrm{dist}\big(y, F(x)\big) < \gamma.$ We will show that 
\begin{eqnarray} \label{PT8}
\mathrm{dist}\big(x, F^{-1}(y)\big) &\leqslant& {\kappa} \mathrm{dist}\big(y, F(x) \big).
\end{eqnarray}
Indeed, since the graph of $F$ is a closed set in $X \times Y,$ there is $y' \in F(x)$ such that 
\begin{eqnarray*}
\|y - y'\|  &=& \mathrm{dist} \big(y, F(x) \big)  \ \leqslant \ \gamma.
\end{eqnarray*}
Then
\begin{eqnarray*}
\|y' - \bar{y}\| &\leqslant& \|y' - y\| + \|y - \bar{y}\| < \gamma + \frac{r}{2} \ < \ r.
\end{eqnarray*}
Since $x \in F^{-1}(y') \cap (X \setminus \mathbb{B}_R),$ we get $x = \phi(y').$ Note that $\phi(y)$ is the unique element of $F^{-1}(y) \cap (X \setminus \mathbb{B}_R).$ Hence
\begin{eqnarray*}
\mathrm{dist} \big(x, F^{-1}(y) \big)  &\leqslant& \|x - \phi(y)\| \ = \ \| \phi(y') - \phi(y)\| \\ 
& \leqslant &  \kappa \|y' - y\| \ = \ \kappa \mathrm{dist} \big(y, F(x) \big),
\end{eqnarray*}
as required.
\end{proof}

We have already seen in Proposition~\ref{MD32} that metric regularity at infinity is preserved under small Lipschitz perturbations.
The following result complements this property for the case of strong metric regularity at infinity.

\begin{proposition}\label{MD38}
Let $X$ and $Y$ be Banach spaces, $F \colon X \rightrightarrows Y$ be a set-valued mapping with closed graph, and $\bar{y} \in J(F).$ If 
$F$ is strongly metrically regular at $(\infty, \bar{y}),$ then for any single-valued mapping $f \colon X \to Y$ such that 
\begin{eqnarray*}
\lim_{\|x\|  \to \infty} \|f(x)\| &=& 0 \textrm{ and } \quad \mathrm{lip}f(\infty) \ < \ \frac{1}{\mathrm{reg} F(\infty, \bar{y})},
\end{eqnarray*}
we have $F + f$ is strongly metrically regular at $(\infty, \bar{y}).$
\end{proposition}

\begin{proof}
Choose $\kappa > 0$ and $\lambda > 0$ such that 
\begin{eqnarray*}
\mathrm{lip} f(\infty)  &<& \lambda \ < \ \kappa^{-1} \ <  \ \frac{1}{\mathrm{reg} F(\infty, \bar{y})}.
\end{eqnarray*}
By Proposition~\ref{BD37}, there exist constants $R > 0$ and $r > 0$ such that the mapping $y \mapsto F^{-1}(y) \cap (X \setminus \mathbb{B}_R)$ is single-valued on $\mathbb{B}_r(\bar{y})$ and Lipschitz continuous with constant $\kappa.$ Denote by $\phi(y)$ the unique element of $F^{-1}(y) \cap (X \setminus \mathbb{B}_R),$ when $y \in \mathbb{B}_r(\bar{y}).$ We have
\begin{eqnarray*}
\| \phi(y') - \phi(y)\| &\leqslant&  \kappa \|y' - y\| \quad \textrm{ for all } \quad y, y' \in \mathbb{B}_r(\bar{y}).
\end{eqnarray*}
Increasing $R$ if necessary, we may assume that 
\begin{eqnarray*}
\|f(x)\| &<& \frac{r}{2} \quad \textrm{ and } \quad \|f(x) - f(x')\| \ \leqslant \ \lambda \|x - x'\|
\end{eqnarray*} 
for all $x, x' \in X \setminus \mathbb{B}_{R}.$

By Proposition~\ref{MD32}, it suffices to show that for each $y \in \mathbb{B}_{\frac{r}{2}}(\bar{y}),$ the set $(F + f)^{-1}(y) \cap \big(X \setminus \mathbb{B}_{R}\big) $ can have at most one element.

On the contrary, assume that $y \in \mathbb{B}_{\frac{r}{2}}(\bar{y})$ and $x^1, x^2 \in X ,$ with $x^1 \ne x^2,$ are such that both $x^1$ and $x^2$ belong to $(F + f)^{-1}(y) \cap \big(X \setminus \mathbb{B}_{R}\big).$ Then, for $i = 1, 2$ we have
$x^i \in F^{-1}(y - f(x^i)) \cap \big(X \setminus \mathbb{B}_{R}\big)$ and 
\begin{eqnarray*}
\|y - f(x^i) - \bar{y}\|  &\leqslant& \|y - \bar{y}\| + \|f(x^i)\| \ < \ \frac{r}{2} + \frac{r}{2} \ = \ r.
\end{eqnarray*}
Therefore, $x^i = \phi(y - f(x^i)).$ Then
\begin{eqnarray*}
0 &<& \|x^1 - x^2\| \ = \ \|\phi(y - f(x^1)) - \phi(y - f(x^2)) \| \\
&\leqslant& \kappa \|f(x^1) - f(x^2)\| \ \leqslant \ \kappa \lambda \| x^1 - x^2\| \ < \ \| x^1 - x^2\|,
\end{eqnarray*}
which is absurd, and we are done.
\end{proof}

Based on Theorem~\ref{MainTheorem1}, it is now easy to obtain a parallel radius result for strong metric regularity.

\begin{proof}[Proof of Theorem~\ref{MainTheorem2}]
This follows directly from Theorem~\ref{MainTheorem1} and Proposition~\ref{MD38} with the observation that strong metric regularity at infinity is a stronger property than metric regularity at infinity. 
\end{proof}

\section{Conclusions} \label{Section5}

Metric regularity and strong metric regularity at infinity of set-valued mappings between Banach spaces are introduced in this paper.
We have shown that these properties are stable under small Lipschitz perturbations.
We have established some relationships between the modulus of metric regularity at infinity and the radius of (strong) metric regularity infinity.

It would be nice to have applications of the results given here in variational analysis and optimization. This will be studied in the future research.

\subsection*{Acknowledgments}
The authors would like to thank Professor Alexander Y. Kruger for the useful discussions.
A part of this work was done while the second author was visiting Academy for Advanced Interdisciplinary Studies, Northeast Normal University, Changchun, China; he is grateful to the Academy for its hospitality and support.


\end{document}